\documentclass[12pt]{article}    
\usepackage[utf8]{inputenc}
\usepackage[T1]{fontenc}
\usepackage{hyperref}
\usepackage{graphicx}
\usepackage{amsmath}
\usepackage{mathtools}
\usepackage{amsfonts}
\usepackage{amssymb}

\begin{document}
\newtheorem{definition}{Definition}
\newtheorem{theorem}{Theorem}
\newtheorem{example}{Example}
\newtheorem{corollary}{Corollary}
\newtheorem{lemma}{Lemma}
\newtheorem{proposition}{Proposition}
\newtheorem{remark}{Remark}
\newenvironment{proof}{{\bf Proof:\ \ }}{\qed}
\newcommand{\qed}{\rule{0.5em}{1.5ex}}
\newcommand{\bfg}[1]{\mbox{\boldmath $#1$\unboldmath}}

\numberwithin{proposition}{section}
\numberwithin{equation}{section} \numberwithin{theorem}{section}
\numberwithin{example}{section} \numberwithin{definition}{section}
\numberwithin{lemma}{section} \numberwithin{remark}{section}

\begin{center}

\section*{Revisiting ``A universal model for the Lorenz curve with novel applications''}

\vskip 0.1in \bf Jos\'e Mar\'{\i}a Sarabia$\,^a$\footnote{Corresponding
author. E-mail address: sarabiaj@unican.es (JM Sarabia), jordav@unican.es (V. Jord\'a), tejeriam@unican.es (M. Tejer\'ia), emilio.gomez-deniz@ulpgc.es (E. G\'omez-D\'eniz)}, Vanesa Jord\'a$\,^a$, \\ Mercedes Tejer\'ia$\,^a$, Emilio G\'omez-D\'eniz$\,^b$
\vskip 0.1in

{\small\it $\,^a$Department of Economics and SANFI Institute,\\ University of Cantabria, Santander, Spain}\\

{\small\it $\,^b$Department of Quantitative Methods in Economics and TIDES Institute, University of Las Palmas de Gran Canaria, Spain }\\[-0.2cm]

\end{center}

\begin{abstract}\noindent
This research reviews several crucial aspects of the universal model for the Lorenz curve proposed by Sitthiyot and Holasut (2023) (hereafter, SH (2023)). A first issue concerns the mathematical definition of the proposed curves. The four functional forms introduced by SH (2023) do not satisfy the necessary and sufficient conditions for a valid Lorenz curve. We propose corrected versions of the previous curves and derive analytical expressions for some measures of inequality. From a theoretical perspective, we rigorously demonstrate that a universal Lorenz curve cannot exist. Therefore, no single parametric form can adequately represent all possible size distributions. This conclusion is reinforced by recent large-scale empirical studies, based on thousands of income distributions in different countries and time periods, showing that different parametric models perform better in different contexts and that no single specification is consistently dominant. Regarding datasets containing zeros, we have shown that the presence of zero observations is a sample-dependent feature that can be accommodated by any genuine Lorenz curve through a standard transformation. Therefore, modeling zeros does not require a specific universal functional form. We have also corrected the erroneous claim made by SH (2023) that the Gini index of the SCS Lorenz curve of Sarabia et al. (1999) \cite{Sarabia1999} does not admit a closed-form expression. Finally, we re-examined the numerical experiments reported by SH (2023). More generally, when evaluating the evidence across all datasets and goodness-of-fit criteria, the proposed model does not consistently demonstrate superiority over existing four-parameter alternatives. Its apparent advantages are limited to the specific edge cases for which it was designed. Taken together, these findings indicate that the model proposed by SH (2023), although potentially useful for datasets with many zeros or extreme single-observation dominance, does not provide a universal solution to the Lorenz curve fitting problem.
\end{abstract}

\noindent {\bf Keywords}: Functional form, Gini index, genuine Lorenz curve

\section{Introduction}\label{s1}

The Lorenz curve and the Gini index are the most important instruments for the analysis of economic inequality \cite{Arnold2018}. The Lorenz curve graphs the cumulative income share $L(p)$ as a function of the cumulative population share $p$ when the income units are arranged according to income size. Using the well-known Gastwirth's definition \cite{Gastwirth1971}, for any cumulative distribution function $F_x(x)$ with support in a subset of the non-negative real numbers and finite and positive expectation $E(X)=\mu_X$, the Lorenz curve of $X$ is defined as,
\begin{equation}\label{basicdeflc}
L_X(u):=\frac{1}{\mu_X}\int_0^u F_X^{-1}(y)dy,
\end{equation}
with $0\le u\le 1$, where $F_X^{-1}(y) = \sup\{x:F_X(x) \leq y\}$, $0 \leq y < 1$, and $\sup\{x:F_X(x) < 1\}$, $y = 1$, is the right-continuous inverse cdf or quantile function. We consider the class of univariate cdfs with $F_X(0)=0$ and $\mathbb{E}(X)>0$, denoted as  ${\cal L}$.

Note that a Lorenz curve $L(p)$ is defined as a function $L:[0,1]\to [0,1]$ that satisfies $L(0)=0$, $L(1)=1$, and is increasing and convex. A characterization of the Lorenz curve attributed to Gaffney and Anstis by Pakes (1981) (see also Arnold and Sarabia, 2018) is given by the following theorem.
\begin{theorem}\label{Theorem1}
Suppose $L(p)$ is defined and continuous on $[0,1]$ with second derivative $L''(p)$. The function $L(p)$ is a Lorenz curve if and only if,
\begin{equation}\label{LCgenuine}
L(0) = 0,\;\;L(1) = 1,\;\;L'(0^+) \geq 0,\;\;L''(p) \geq 0, \quad p \in (0,1),
\end{equation}
\end{theorem}

Several parametric models have been proposed for the Lorenz curve. Among the most notable are those of Kakwani and Podder \cite{kakwani1973}, Rasche et al. \cite{rasche1980}, Pakes \cite{pakes1995}, Aggarwal and Singh \cite{aggarwal1984}, Arnold \cite{arnold1986}, Villase\~{n}or and Arnold \cite{villasenor1989}, Ortega et al.\cite{ortega1991}, Chotikapanich \cite{chotikapanich1993}, and Sarabia et al. \cite{sarabia2026}, among others.

More recently, in an article published in this journal, Sitthiyot and Holasut \cite{Sitthiyot2023} (hereafter, SH (2023)) proposed a universal model for the Lorenz curve, with applications to datasets containing zeros or exhibiting extreme inequality. In the present paper, we examine a number of relevant aspects of that contribution.

The remainder of the paper is organized as follows. In Section \ref{s2}, we show that the curves proposed by SH (2023) are not genuine Lorenz curves. We also propose new modified versions of these curves that satisfy the defining properties of a genuine Lorenz curve, and we derive several inequality measures, including the Gini index. In Section \ref{s3}, we examine the claim that a universal model for the Lorenz curve exists and refute it on several grounds. We also discuss the modeling of datasets containing zeros by means of Lorenz curves. Section \ref{s4} analyzes the Gini index of the Sarabia et al. \cite{Sarabia1999} Lorenz curve. Relevant aspects of the numerical experiments reported by SH (2023) are discussed in Section \ref{s5}. Finally, Section \ref{s6} presents the main conclusions.

\section{The SH (2023) curves and new proposals}\label{s2}

In this section, we show that the curves defined by Equations (1)–(4) in SH (2023) are not correctly specified and, in general, are not genuine Lorenz curves; a formal proof of this result is given below. They will be denoted by $y_k(x)$, for $k=1,2,3,4$, corresponding to Equations (1)–(4) in SH (2023).

As a preliminary step, we introduce the Lorenz curves proposed by Sarabia, which motivated the formulation in SH (2023). The family of curves proposed by SH (2023) is inspired by the parametric family of Lorenz curves introduced by Sarabia \cite{Sarabia1997}. In that paper, Sarabia proposed a new parametric family of Lorenz curves based on the generalized Tukey lambda distribution. This family of curves is presented in the Appendix. Sarabia's aim was to construct a hierarchical family of curves with an increasing number of parameters, including the egalitarian curve, the power Lorenz curve, and the Lorenz curve of the classical Pareto distribution. This idea was later adopted by SH (2023), with the addition of an extra parameter.

The curve $y_1(x)$ (linear function), corresponding to equation (1) in SH (2023), is defined as,
\begin{equation}\label{extcurvey1}
y_1(x)=
    \begin{dcases}
        0 & \text{if}\,\, x-\delta<0\iff x<\delta, \\
        \frac{2}{P+1}\left(\frac{x-\delta}{1-\delta}\right) & \text{if}\,\, 0\le x-\delta<1\iff \delta\le x<1+\delta,\\
        1 & \text{if}\,\, x=1, \\
    \end{dcases}
\end{equation}
where $0\le \delta<1$ and $P\ge 1$. The curve $y_2(x)$ (exponential function) is defined by
\begin{equation}\label{extcurvey2}
y_2(x)=
    \begin{dcases}
    0 & \text{if}\,\, x-\delta<0\iff x<\delta, \\
    \left(\frac{x-\delta}{1-\delta}\right)^P & \text{if}\,\, 0\le x-\delta<1\iff \delta\le x<1+\delta,\\
    \end{dcases}
\end{equation}
where $0\le \delta<1$ and $P\ge 1$. The functional form $y_3(x)$ is given by,
\begin{equation}\label{extcurvey3}
y_3(x)=
    \begin{dcases}
    0 & \text{if}\,\, x-\delta<0\iff x<\delta, \\
    1-\left(1-\left(\frac{x-\delta}{1-\delta}\right)\right)^{1/P} & \text{if}\,\, 0\le x-\delta<1\iff \delta\le x<1+\delta,\\
    \end{dcases}
\end{equation}
Finally, $y_4(x)$ is given by
\begin{eqnarray}\label{extcurvey4}
y_4(x) &=& (1-\rho)\frac{2}{P+1}\left(\frac{x-\delta}{1-\delta}\right)+\rho\left[(1-w)\left(\frac{x-\delta}{1-\delta}\right)^p\right.\nonumber\\
&&+
\left.w\left(1-\left(1-\left(\frac{x-\delta}{1-\delta}\right)\right)^{1/P}\right)\right].
\end{eqnarray}

\begin{proposition}\label{proposition1}
The curves (\ref{extcurvey1}) to (\ref{extcurvey4}) proposed by SH (2023) are not genuine LCs.
\end{proposition}
\begin{proof}
The domain and the range of the curve $y_1(x)$ are not correctly defined since $x\in [0,1+\delta]$ with $0\le\delta<1$. On the other hand, $y_1(1)=\frac{2}{P+1}$ and since $y_1(1)=1$ (second condition in (\ref{LCgenuine})) we have $P=1$, that is, the parameter $P$ cannot be modeled.

The domain and range of the two curves $y_2(x)$ and $y_3(x)$ are not well defined since $x\in [0,1+\delta]$. Again, the domain of the curve $y_4(x)$ is incorrectly defined ($x\in [0,1+\delta]$). Likewise, $y_4(1)=(1-\rho)\frac{2}{P+1}+\rho$ and since $y_4(1)=1$, we have the constraint
\begin{equation}\label{constraint2}
\frac{2(1-\rho)}{P+1}+\rho=1,
\end{equation}
which should be added. The following extreme cases deserve attention, as they will be used later to discuss the numerical experiments of SH (2023). Using (\ref{constraint2}), if $\rho=1$, the parameter $P$ disappears in $y_4(x)$, and then $P$ cannot be modeled. On the other hand, if $\rho=0$ then $y_4(x)=y_1(x)$ and in consequence $P=1$.
\end{proof}\\
\subsection{New and genuine Lorenz curves}

In this section, we revisit the Lorenz curves proposed by SH (2023), given in (\ref{extcurvey1})--(\ref{extcurvey4}), and derive corrected versions of these models. We also obtain explicit expressions for several inequality indices. In light of the discussion in the previous section and by applying Theorem~\ref{Theorem1}, the corrected Lorenz curves are given by:
\begin{equation}\label{newL1}
L_1(x;\delta)=
    \begin{dcases}
        0 & \text{if } 0\le x\le \delta,\\
        \frac{x-\delta}{1-\delta} & \text{if } \delta\le x\le 1,
    \end{dcases}
\end{equation}
with $0\le x\le 1$ and $0\le\delta< 1$. The second curve, which is a power Lorenz curve, is defined by,
\begin{equation}\label{newL2}
L_2(x;\delta,\alpha_1)=
    \begin{dcases}
        0 & \text{if } 0\le x\le \delta,\\
        \left(\frac{x-\delta}{1-\delta}\right)^{\alpha_1} & \text{if } \delta\le x\le 1,
    \end{dcases}
\end{equation}
where $0\le x\le 1$, $\alpha_1\ge 1$, and the curve $L_3$ is given by,
\begin{equation}\label{newL3}
L_3(x;\delta,\alpha_2)=
    \begin{dcases}
        0 & 0\le x\leq \delta, \\
        1-\left(1-\frac{x-\delta}{1-\delta}\right)^{\alpha_2} & \delta\leq x\leq 1, \\
    \end{dcases}
\end{equation}
with $0\le x\le 1$, where $0< \alpha_2\le 1$. Finally, a corrected and more general version of (\ref{extcurvey4}) is given by the following convex combination:
\begin{equation}\label{newL4}
L_4(x;\delta,\alpha_1,\alpha_2)=\pi_1L_1(x;\delta)+\pi_2L_2(x;\delta,\alpha_1)+(1-\pi_1-\pi_2)L_3(x;\delta,\alpha_2),
\end{equation}
where $\pi_1,\pi_2>0$ and $\pi_1+\pi_2<1$.

Note that models (\ref{newL2})--(\ref{newL4}) are more general than those proposed by SH (2023) (after correction), since the parameters $\alpha_1$ and $\alpha_2$ are unrelated, and curve (\ref{newL4}) includes an additional shape parameter. If we take $\alpha_1=1/\alpha_2$ in (\ref{newL4}), we recover the original model (after correction) proposed by SH (2023). Figure \ref{Figure1} shows the Lorenz curves $L_1(x;\delta)$ (equation (\ref{newL1})) and $L_2(x;\delta,\alpha_1)$ (equation (\ref{newL2})) for selected parameter values. Figure \ref{Figure2} also shows the $L_3(x;\delta,\alpha_2)$ Lorenz curve (equation (\ref{newL3})) for some selected values of the parameters.

Note that the curve $L_2$ is ordered with respect to the Lorenz order as a monotone function of the parameters $\delta$ and $\alpha_1$. The same occurs with the curve $L_3$ with respect to the parameters $\delta$ and $\alpha_2$. A similar phenomenon occurs with the curve $L_4$.

\begin{figure}[ht]
\centering
\includegraphics[width=\linewidth]{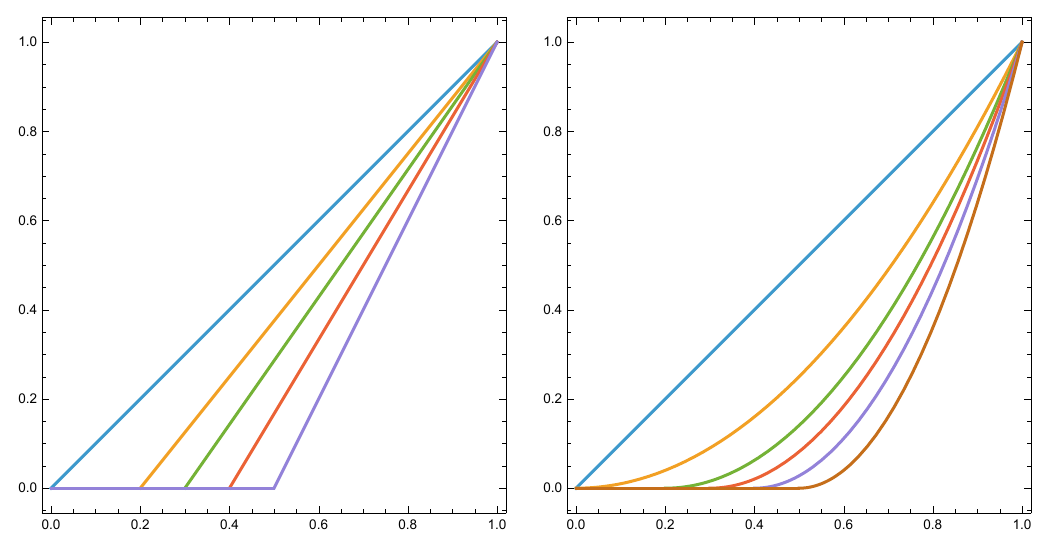}
\caption{Lorenz curves $L_1$ (left), equation (\ref{newL1}) for $\delta\in\{0.0, 0.2,0.3,0.4,0.5\}$ (left to right) and curve $L_2$ (right), equation (\ref{newL2}), for $\alpha_1=2$ and the same $\delta $ values and including $L(p)=p$.}
\label{Figure1}
\end{figure}

\begin{figure}[ht]
\centering
\includegraphics[width=\linewidth]{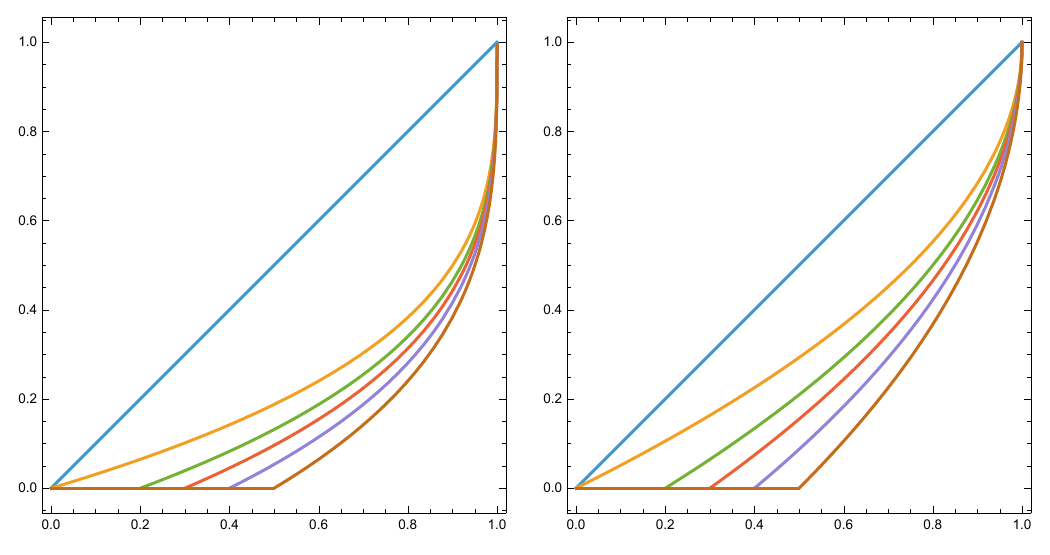}
\caption{Lorenz curves $L_3$ (equation (\ref{newL3})) for $\alpha_2=0.3$ (left) and $\alpha_2=0.5$ (right) for $\delta\in\{0.0, 0.2,0.3,0.4,0.5\}$ (left to right), including $L(p)=p$.}
\label{Figure2}
\end{figure}

\subsection{Inequality indices}

For the curves above, many inequality indices can be obtained in closed form. Except for the curve given in their equation (1), SH (2023) do not obtain these closed-form expressions.
First, the Gini index of the curve (\ref{newL1}) is,
\begin{equation*}
G_1=2\int_{0}^{1}[x-y_1(x)]dx=1-2\int_0^1y_1(x)dx=\delta.
\end{equation*}
An important property of the curves (\ref{newL1})--(\ref{newL3}) is that their corresponding Gini indices satisfy $\delta \leq G_i \leq 1$, $i=1,2,3$. Here, the parameter $\delta$ has a natural interpretation as the proportion of zero observations in the variable of interest. In this sense, $\delta$ reflects the structural presence of zeros in the distribution and determines the minimum level of inequality compatible with the model. Figure \ref{Figure3} shows the Gini indices of the $L_2$ and $L_3$ curves for some selected parameter values.

\begin{figure}[ht]
\centering
\includegraphics[width=\linewidth]{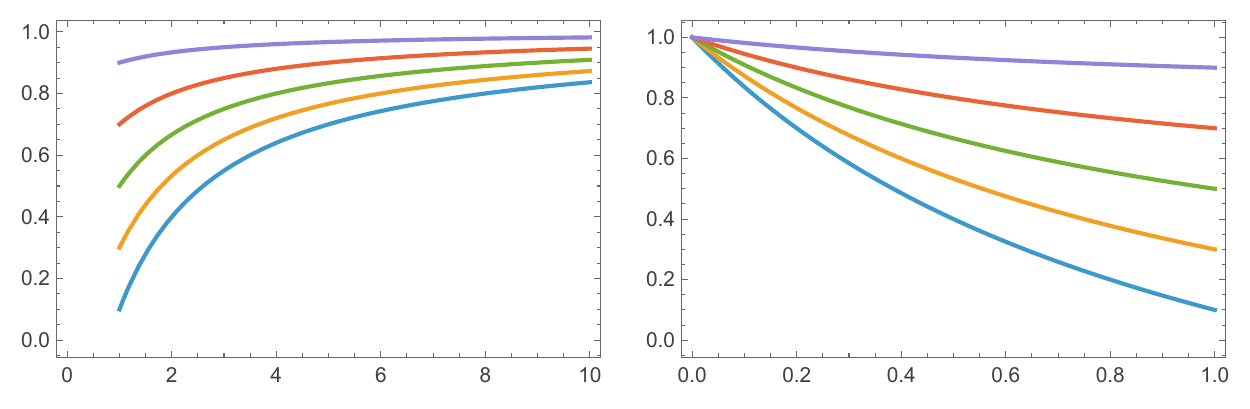}
\caption{Gini index of the curve $L_2(p;\delta,\alpha_1)$ as a function of $\alpha_1$, for $\delta\in\{0.1,0.3,0.5,0.7,0.9\}$ (left) and Gini index of the curve $L_3(p;\delta,\alpha_2)$ as a function of $\alpha_2$ for the same $\delta$ values (right)}
\label{Figure3}
\end{figure}

Finally, the Gini index of the curve $y_4(x)$ is a linear convex combination of the Gini indices and is given by,
\begin{equation*}
G_4=\pi_1\delta+\pi_2\frac{\alpha_1-1+2\delta}{1+\alpha_1}+(1-\pi_1-\pi_2)\frac{1-\alpha_2+2\alpha_2\delta}{1+\alpha_2}.
\end{equation*}

An important generalization of the Gini index was proposed by Donaldson and Weymark \cite{donaldson1980} and Kakwani \cite{kakwani1980}, and later studied in detail by Yitzhaki \cite{yitzhaki1983}. These authors introduced the generalized Gini index, defined as
\begin{equation}\label{DWKY}
G(\nu)=1-\nu(\nu+1)\int_0^1(1-p)^{\nu-1}y(p)\,dp,
\end{equation}
where $\nu>1$ and $y(\cdot)$ is a Lorenz curve. If we set $\nu=1$ in \eqref{DWKY}, we recover the standard Gini index. Moreover, as $\nu$ increases, greater weight is assigned to the lower tail of the distribution. In the limit, as $\nu \to \infty$, the index depends only on the minimum income, which is consistent with the Rawlsian criterion that social welfare depends exclusively on the poorest member of society. Table \ref{TableGenGini} includes the Gini and the DWK indices for the curves $L_1$, $L_2$, and $L_3$. The derivation of these expressions is direct.

\begin{table}[ht]
\centering
\begin{tabular}{|c|l|l|}
\hline
Lorenz curve & Gini index & Generalised Gini index \\
\hline
$L_1(x;\delta)$ & $G_1=\delta$ &$G_1(\nu)=1-(1-\delta)^\nu$ \\
$L_2(x;\delta,\alpha_1)$ & $G_2=\frac{\alpha_1+2\delta-1}{1+\alpha_1}$ &$G_2(\nu)=1-\frac{\nu(\nu+1)(1-\delta)^\nu\Gamma(\nu)(\alpha_1+1)}{\Gamma(1+\alpha_1+\nu)}$ \\
$L_3(x;\delta,\alpha_2)$ & $G_3=\frac{1-\alpha_2+2\alpha_2\delta}{1+\alpha_2}$ &$G_3(\nu)=1-\frac{(1+\nu)\alpha_2(1-\delta)^\nu}{\alpha_2+\nu}$ \\
\hline
\end{tabular}
\caption{Inequality measures for the curves $L_1$, $L_2$ and $L_3$}
\label{TableGenGini}
\end{table}

\section{Is there a universal model for representing a Lorenz curve?}\label{s3}

In this section, we answer the question of whether there is a universal model for representing a Lorenz curve. The answer to this question can be obtained through the following proposition.

\begin{proposition}
Let $X$ be a positive random variable with support on an interval of $\mathbb{R}^+$, strictly positive mean $\mu_X$ and cumulative distribution function $F_X(x)$. Then, the Lorenz curve $L_X(p)$ of $X$ determines the distribution function $F_X(x)$ up to a scale factor.
\end{proposition}
\begin{proof}
The Lorenz curve of a random variable $X$ is given by the equation (\ref{basicdeflc}). Calculating the first derivative, we have,
\begin{equation*}
L_X'(u)=\mu_XF_X^{-1}(u),\;\;0\le u\le 1,
\end{equation*}
and obviously $F_X^{-1}(u)$ determines $F_X(u)$, up to a scale factor, which proves the result.
\end{proof}

The previous result rigorously shows that it is not possible to obtain a universal Lorenz curve for all random variables, since each Lorenz curve is associated with a distribution function, up to a scale factor. The fact that Lorenz curves often display a similar overall shape does not imply that all types of data can be adequately represented by a single universal curve.

The proposal of SH (2023), if correctly specified as in (\ref{extcurvey1})--(\ref{extcurvey4}), may be interpreted as a flexible model that also allows for the modeling of datasets containing zeros. However, this does not imply that these curves constitute universal models. The presence of zeros in a dataset is a sample-specific feature that can, in principle, be accommodated within any parametric Lorenz curve as follows. If we have a dataset of size $n$, where $n_0$ are zeros, and $L(p)$ is the initial Lorenz curve, we can define the new Lorenz curve $L_0(p)$ as,
$$
L_0(p)=
\begin{dcases}
        0, & \text{if } 0\leq p\leq p_0,\\
        \dfrac{L(p)-L(p_0)}{1-L(p_0)}, & \text{if } p_0\leq p\leq 1,
\end{dcases}
$$
where $p_0=n_0/n$ is the proportion of zeros.

The absence of a universal functional form for the Lorenz curve can also be examined from an empirical perspective. Clearly, the answer depends to a large extent on the type of data considered and on the metric used for model comparison. In this regard, we rely on two recent studies of particular relevance.

In the first study, Blesch et al. \cite{Blesch2022} use a uniquely fine-grained dataset comprising $n=3,056$ U.S. county-level income distributions and estimate and compare a total of 17 parametric Lorenz curve models. To evaluate the competing specifications, the authors employ the Borda metric\cite{brams2002}, which ranks alternatives by assigning points according to their positions in the ordering and then summing these points across comparisons. The model with the highest Borda score is selected as the preferred specification. Among the 17 functional forms considered, the three best-performing models are the Lorenz curve proposed by Ortega, the Lorenz curve associated with the GB2 distribution, and, in third place, the Lorenz curve corresponding to the Dagum distribution.

The second study, by Jord\'a et al. \cite{Jorda2021}, also provides evidence against the existence of a universal Lorenz curve. Using more than 3,300 grouped income datasets covering over 180 countries during the period 1867--2015, the authors show that no single parametric form dominates across all income distributions. Their systematic comparison of distributions within the generalized beta distribution of the second kind (GB2) family---including the Singh--Maddala, Dagum, beta 2, lognormal, and Fisk distributions---reveals that the preferred model varies substantially across datasets. Although the GB2 distribution provides the best overall fit in the majority of cases according to information criteria, three-parameter submodels are preferred in approximately 22--37\% of the cases, and no single curve consistently outperforms the others in all empirical settings. This variability directly contradicts the premise of a universal functional form: if income data alone, even within a well-studied distributional family, require different specifications depending on the country and time period, then the claim that a single parametric curve can adequately represent all size distributions across disciplines is implausible.

\section{The Gini index of the SCS Lorenz curve}\label{s4}

The next class of Lorenz curve,
\begin{equation}\label{SCSLC}
L(p;\alpha,\beta,\gamma)=p^\gamma\left(1-(1-p)^\alpha\right)^\beta,\;\;0\le p\le 1,
\end{equation}
was proposed by Sarabia et al. \cite{Sarabia1999} in (1999). This curve will be called SCS LC. Several properties of this family, including Lorenz orderings, estimation, and calculation of inequality measures, were studied by these authors. On page 5 of SH (2023), the authors state: ``\textit{Note that the SCS model does not have an explicit mathematical solution for the Gini index}''. Likewise, on page 8, they write: ``\textit{the Gini index based on the SCS model does not exist\dots}''. However, these statements are incorrect. An explicit analytical expression for the Gini index is available (see \cite{Sarabia1999}, p.~57, formula (A5)),

\begin{equation}\label{GiniSCS}
G(\alpha,\beta,\gamma)=1-2\sum_{j=0}^{\infty}\frac{\Gamma(j-\beta)}{\Gamma(j+1)\Gamma(-\beta)}B(\gamma+1,\alpha j+1),
\end{equation}
where $B(x,y)=\frac{\Gamma(x)\Gamma(y)}{\Gamma(x+y)}$ is the usual beta function. In the special case $\gamma=0$ in (\ref{SCSLC}) we have the Lorenz curve,
$$
L(p;\alpha,\beta)=\left(1-(1-p)^\alpha\right)^\beta,\;\;0\le p\le 1,
$$
whose Gini index is,
\begin{equation}\label{GiniSCS0}
G(\alpha,\beta)=1-2\frac{\Gamma\left(1+\frac{1}{\alpha}\right)\Gamma(1+\beta)}{\Gamma\left(1+\frac{1}{\alpha}+\beta\right)}.
\end{equation}

Equations~(\ref{GiniSCS}) and~(\ref{GiniSCS0}) can be evaluated without substantial difficulty using either symbolic software, such as Mathematica or Maple, or numerical software, such as R or Python. Therefore, contrary to the claims made by the authors throughout the paper, the computation of the SCS Gini index does not appear to pose serious practical difficulties. This is also supported by Table~6 in SH (2023), where the reported values for the SCS model do not exhibit substantial differences from those obtained for the other models.

\section{The numerical experiments of SH}\label{s5}

The numerical results reported by SH (2023) offer a more nuanced picture than the authors suggest. When the proposed model is compared with the SCS model across all eight datasets and five goodness-of-fit criteria, the two models perform similarly overall, with neither showing a consistent advantage. When the proposed model is instead compared with the four-parameter model of Sarabia \cite{Sarabia1997}, which is a natural benchmark given the equal number of parameters, Sarabia's model performs better in the majority of cases according to the goodness-of-fit statistics considered. The clearest advantage of the models proposed by SH (2023) arises for datasets containing a large proportion of zeros and for the hypothetical dataset constructed to represent extreme single-observation dominance, which are precisely the settings for which the model was specifically designed. Outside these cases, the results do not indicate a systematic improvement over existing parametric forms.

\begin{proposition}
The six models proposed in Table 2 of SH (2023) are not valid.
\end{proposition}
\begin{proof}
This result follows from Proposition \ref{proposition1}. First, if $\rho=1$, the parameter $P$ disappears; consequently, Models 1, 2, 3, 5, 6, and 7 are not valid. If $\rho=0$ (Model 8), then $P=1$, and Model 8 is not valid. Finally, for Model 4, if $\rho=0.88$ and $P=13.42$, then constraint (\ref{constraint2}) is not satisfied.
\end{proof}

\section{Conclusions}\label{s6}
This article has reviewed several aspects of the universal model for the Lorenz curve proposed by SH (2023). Our analysis has identified a number of issues with the original proposal and has provided several clarifications, corrections, and extensions of the proposed curves and their associated inequality measures.

A first observation concerns the mathematical definition of the proposed curves. The four functional forms introduced by SH (2023) do not satisfy the necessary and sufficient conditions for a valid Lorenz curve. We have proposed corrected versions, denoted by $L_1$, $L_2$, and $L_3$, together with their convex combination $L_4$, and have derived closed-form expressions for the Gini index and the generalized Donaldson--Weymark--Kakwani index. The corrected model $L_4$ contains the original proposal as a special case and allows for a wider range of distributional shapes.

From a theoretical perspective, we have shown rigorously that a universal Lorenz curve cannot exist. Since the Lorenz curve of a random variable determines its distribution function up to a scale factor, different data-generating processes necessarily produce different Lorenz curves. Hence, no single parametric form can adequately represent all possible size distributions. This conclusion is reinforced by recent large-scale empirical studies based on thousands of income distributions across countries and time periods, which show that different parametric models perform best in different settings and that no single specification consistently dominates.

With regard to datasets containing zeros, we have shown that the presence of zero observations is a sample-dependent feature that can be accommodated by any genuine Lorenz curve through a standard transformation. Therefore, the modeling of zeros does not require a dedicated universal functional form.

We have also corrected the erroneous claim made by SH (2023) that the Gini index of the SCS Lorenz curve of Sarabia et al. \cite{Sarabia1999} does not admit a closed-form expression. In fact, an explicit analytical formula exists and can be evaluated without difficulty using standard symbolic or numerical software.

Finally, we have re-examined the numerical experiments reported by SH (2023). Several of the datasets considered are known to follow power laws, in which case the appropriate Lorenz curve is the Pareto curve, rendering comparisons with other models unnecessary for those examples. More generally, when the evidence is assessed across all datasets and goodness-of-fit criteria, the proposed model does not exhibit consistent superiority over existing four-parameter alternatives. Its apparent advantages are confined to the specific boundary cases for which it was designed. Taken together, these findings indicate that the model proposed by SH (2023), although potentially useful for datasets with many zeros or extreme single-observation dominance, does not provide a universal solution to the problem of fitting Lorenz curves.

\section*{Acknowledgements}

\subsection*{Author contributions}
All authors contributed equally to this work.

\subsection*{Funding}
JMS, VJ and MT acknowledge financial support from the I+D+i project Ref. PID2024-156871NB-I00 financed by MICIU/AEI/10.13039/501100011033/FEDER, UE. EGD was partially funded by grant PID2021-127989OB-I00 (Agencia Estatal de Investigación, Ministerio de Ciencia e Innovación, Spain).

\subsection*{Data availability}
This study did not use any dataset, as the results are entirely theoretical.

\subsection*{Competing interests}

The authors declare no competing interests.

\section*{Appendix}

Sarabia's (1997) family of Lorenz curves is given by,
\begin{eqnarray*}
  L_0(p) &=& p, \\
  L_1(p;\alpha_1) &=& \pi_1p+(1-\pi_1)p^{\alpha_1}, \\
  L_2(p;\alpha_2) &=& \pi_1p+(1-\pi_1)(1-(1-p)^{\alpha_2}), \\
  L_3(p;\alpha_1,\alpha_2)&=&\pi_1p+\pi_2p^{\alpha_1}+(1-\pi_1-\pi_2)(1-(1-p)^{\alpha_2}),
\end{eqnarray*}
where $\alpha_1 \geq 1$, $0 < \alpha_2 \leq 1$, and $\pi_i$ are non-negative weights that sum to one. The main contribution is the introduction of a methodology for constructing families of Lorenz curves with an increasing number of parameters, which allows classical curves (such as the Pareto curve) to be interpreted as special or limiting cases. The Tukey distribution is a well-known and highly flexible family of distributions, and it provides a better representation of empirical income and wealth data, especially with regard to tail behavior and when prior information about the data is unavailable. In addition, the paper derives closed-form expressions for widely used inequality measures, such as the Gini and Pietra indices, which facilitates their computation from the model parameters.


\end{document}